\documentclass[11pt,a4paper]{article}

\usepackage[utf8]{inputenc}
\usepackage[T1]{fontenc}
\usepackage{lmodern}
\usepackage{microtype}

\usepackage[margin=2.5cm]{geometry}

\usepackage{amsmath,amssymb,amsthm,mathtools}

\usepackage{graphicx}
\usepackage[dvipsnames]{xcolor}
\usepackage{enumitem}
\usepackage{booktabs}
\usepackage{cite}

\usepackage{hyperref}

\usepackage[capitalise,noabbrev]{cleveref}

\hypersetup{
    colorlinks=true,
    linkcolor=NavyBlue,
    citecolor=NavyBlue,
    urlcolor=NavyBlue,
    pdfauthor={Daniel Debrohim, Diana Sasaki, Patrícia Nunes},
    pdftitle={Connectedness and Enumeration of Reconfiguration Graphs of Cyclically Colored Triangulations}
}

\setlist[itemize]{noitemsep, topsep=4pt, leftmargin=*}
\setlist[enumerate]{noitemsep, topsep=4pt, leftmargin=*}

\theoremstyle{plain}
\newtheorem{theorem}{Theorem}[section]
\newtheorem{lemma}[theorem]{Lemma}
\newtheorem{proposition}[theorem]{Proposition}

\theoremstyle{definition}
\newtheorem{definition}[theorem]{Definition}

\theoremstyle{remark}
\newtheorem{remark}[theorem]{Remark}

\newcommand{\ccl}{c}

\DeclareMathOperator{\ind}{ind}

\title{\Large\bfseries Cyclically Colored Triangulations: Enumeration and Connectedness of Reconfiguration Graphs}

\author{
  Daniel Debrohim\thanks{Institute of Mathematics and Statistics, Rio de Janeiro State University (UERJ), Rio de Janeiro, Brazil. Email: \texttt{daniel.debrohim@pos.ime.uerj.br}} \and
  Diana Sasaki\thanks{Institute of Mathematics and Statistics, Rio de Janeiro State University (UERJ), Rio de Janeiro, Brazil. Email: \texttt{diana.sasaki@ime.uerj.br}} \and
  Patr\'icia Nunes\thanks{Institute of Mathematics and Statistics, Rio de Janeiro State University (UERJ), Rio de Janeiro, Brazil. Email: \texttt{nunes@ime.uerj.br}}
}

\date{}

\begin{document}

\maketitle

\begin{abstract}
We study the connectedness and enumeration of reconfiguration graphs of valid triangulations of convex polygons whose vertices are cyclically colored with $j \ge 3$ colors, where every triangle has vertices of three pairwise distinct colors.

For $j = 3$, we settle a conjectural expectation of Acharya,
M\"utze, and Verciani: we prove that the twist graph $\mathcal{H}_{3k+2}$ is connected for every $k \ge 4$, whereas $\mathcal{H}_8$ and $\mathcal{H}_{11}$ are disconnected. Using a colored root-edge decomposition that induces Cartesian products in the state space, we obtain coupled recurrences for $T(3k)$ and $T(3k+2)$. The corresponding generating functions reduce to the equation $U(x) = 1 + xU(x)^4$, and the difference between the two consecutive families is given by the Raney number $T(3k+3) - T(3k+2) = R_{4,5}(k-1)$.

For $j \ge 4$, reconfiguration is performed by validity-preserving diagonal flips. We extend the root-edge decomposition to all admissible classes $N \not\equiv 1 \pmod{j}$, obtaining, for each fixed $j$, a finite algebraic system of functional equations. We further prove that the flip graph $\mathcal{G}_N^{(j)}$ is connected whenever valid triangulations exist. Thus, the root-edge decomposition provides a unified structural framework for the enumeration and reconfiguration of cyclically colored triangulations.
\end{abstract}
\medskip
\noindent\textbf{Keywords:}
colored triangulations; reconfiguration graphs; connectedness;
twists; flips; root-edge decomposition; generating functions;
Raney numbers.
\section{Introduction}
\label{sec:intro}

Triangulations of convex polygons have been studied since the work of Euler. Their enumeration leads to the Catalan numbers: the number of triangulations of a convex polygon with $n + 2$ vertices is the Catalan number $C_n$. A classical way to obtain their recurrence is to fix a boundary edge, distinguish the unique triangle incident with it, and observe that the other two sides of this triangle split the polygon into two independently triangulated subpolygons. This is the decomposition underlying Segner's recurrence
\[
C_{n+1} = \sum_{i=0}^{n} C_i C_{n-i},
\]
introduced in Segner's original work on polygon triangulations~\cite{segner1761}; see also Stanley~\cite{stanley1999}. We shall refer to it as the \emph{root-edge decomposition}. Beyond convex polygons, triangulations and their local transformations arise naturally in discrete geometry, planar point sets, and triangulated surfaces~\cite{lawson1972,hurtado1999,ItoEtAl2022}.

The set of triangulations can itself be organized as a graph whose vertices are triangulations and whose edges correspond to diagonal flips. In the convex case, this flip graph is the $1$-skeleton of the associahedron~\cite{stasheff1963,lee1989,devadoss1999}. It is connected~\cite{lawson1972}, Hamiltonian~\cite{lucas1987}, and its diameter is $2N - 10$ for $N > 12$~\cite{sleator1988,pournin2014}. Such reconfiguration structures are closely tied to Gray codes and the algorithmic generation of combinatorial objects~\cite{vandenheuvel2013,nishimura2018,mutze2023gray}.

Color constraints restrict the set of admissible triangulations and may also change the appropriate local reconfiguration operation. Sagan~\cite{sagan2008}, motivated by questions of Propp, enumerated proper partitions of cyclically colored polygons. In the specialization to triangulations with three colors, properness means that every triangle contains all three colors. This is precisely our validity condition in the three-color case and is equivalent to requiring that no edge of the triangulation be monochromatic.

For $j \ge 4$, however, the framework considered here is different from Sagan's original notion of a proper partition. We continue to work with triangulations, so every face has three vertices, while the polygon uses $j > 3$ colors. Consequently, a triangular face cannot contain all $j$ colors. Instead, we call a triangulation valid when every triangle has three pairwise distinct vertex colors. This agrees with the extension suggested by Sagan for the regime in which the number of colors exceeds the size of a face, but it is not a proper $k$-partition in his original sense. In this $j \ge 4$ regime, the reconfiguration graph is defined using ordinary diagonal flips that preserve validity.

There is also an important distinction between the boundary-coloring conventions. Sagan uses a shifted size parameter: his quantity $b_s$ counts triangulations of a polygon with $s + 2$ vertices, whereas our parameter $N$ is the actual number of vertices, so
\[
T(N) = b_{N-2}.
\]
Moreover, when $N \equiv 1 \pmod{3}$, the genuine cyclic coloring forces the closing boundary edge $v_N v_1$ to be monochromatic. Sagan obtains a nonzero class in the corresponding shifted family by modifying the color assigned to the last vertex. We retain the periodic coloring throughout the entire boundary. Consequently, no valid triangulation exists when
\[
N \equiv 1 \pmod{3}.
\]

In the three-color regime, a diagonal flip does not preserve
validity. Acharya, M\"utze, and Verciani~\cite{AcharyaMutzeVerciani2025} introduced a validity-preserving local operation, called a \emph{twist}, supported on a local hexagon. They defined the twist graph $\mathcal{H}_N$, proved its connectedness when $3 \mid N$, and suggested that $\mathcal{H}_N$ should be disconnected when
$N \equiv 2 \pmod{3}$. We show that this disconnection is confined to the two initial cases $N = 8$ and $N = 11$.

\begin{theorem}\label{thm:intro-main}
The twist graph $\mathcal{H}_{3k+2}$ is connected for every $k \ge 4$.
\end{theorem}

Together with the result of Acharya, M\"utze, and Verciani for $3 \mid N$, and with the direct analysis of the exceptional graphs $\mathcal{H}_8$ and $\mathcal{H}_{11}$, this gives a complete connectedness classification of the cyclically $3$-colored twist graphs.

Our proof starts from the classical root-edge decomposition. In the colored setting, the color of the third vertex of the triangle incident with the distinguished boundary edge is constrained, and its congruence class determines the two resulting subproblems. We use this observation in two complementary ways. First, it partitions $\mathcal{H}_N$ into base-triangle classes isomorphic to Cartesian products of smaller twist graphs; these classes are then joined by explicit local twists. Second, it produces coupled recurrences for $T(3k)$ and $T(3k+2)$. Thus, the classical Catalan decomposition itself is not new. The point is that, under the periodic three-coloring, it simultaneously explains the structure of the reconfiguration graph and the mutual enumerative dependence of the two admissible residue classes.

The same root-edge principle also applies to cyclic colorings with $j \ge 4$ colors, but now the local operation is a validity-preserving diagonal flip. Writing
\[
N = jk + r, \qquad 0 \le r \le j - 1,
\]
we derive a general recurrence for every admissible residue class $r \ne 1$. At graph level, each term of the recurrence corresponds to a base-triangle class inducing a Cartesian product of two smaller restricted flip graphs. This product structure leads to our second connectedness result.

\begin{theorem}\label{thm:intro-jge4}
Let $j \ge 4$ and let $N = jk + r$, where $0 \le r \le j - 1$ and $r \ne 1$. Then the  flip graph $\mathcal{G}_N^{(j)}$  is connected.
\end{theorem}

Equivalently, for every $j \ge 4$, the graph $\mathcal{G}_N^{(j)}$ is connected whenever valid triangulations exist. These results show that the root-edge decomposition provides a common structural mechanism for two distinct reconfiguration regimes: twists for three colors and validity-preserving flips for four or more colors.

The remainder of the paper is organized as follows. In \Cref{sec:prelim}, we introduce valid triangulations, local twists, and the twist graph. In \Cref{sec:base-decomposition}, the base-triangle partition and the product structure of \Cref{lem:product} are used to prove \Cref{thm:intro-main}, together with the exceptional disconnected cases in \Cref{prop:exceptions}. In \Cref{sec:enumeration}, we derive the coupled recurrences of \Cref{prop:coupled}, relate the coupled recurrences to Sagan's closed formulas
in \eqref{eq:sagan-closed-forms}, and obtain the Raney-number identity of \Cref{thm:gap}. Finally, in \Cref{sec:more-colors}, we derive the general recurrence of \Cref{lem:recurrence-general} and prove the connectedness result stated in \Cref{thm:intro-jge4}.

\section{Preliminaries}
\label{sec:prelim}

\begin{definition}[Cyclic \(j\)-coloring]
\label{def:cyclic-coloring}
Let \(P_N=(v_1,v_2,\ldots,v_N)\) be a convex polygon. A
\emph{cyclic \(j\)-coloring} of \(P_N\) is a coloring
\[
c:V(P_N)\longrightarrow \{0,1,\ldots,j-1\}
\]
defined by
\[
c(v_i)=i-1 \pmod j.
\]
\end{definition}
In the three-color case, we shall occasionally write
\[
A:=0,\qquad B:=1,\qquad C:=2.
\]
\begin{definition}[Valid triangulation]
\label{def:valid-triangulation}
A triangle is called \emph{valid} if its three vertices have pairwise distinct colors. A triangulation of $P_N$ is called \emph{valid} if all of its triangles are valid.
\end{definition}

\begin{definition}[Flip]
Let $T$ be a triangulation of a convex polygon. A \emph{flip} is a local operation performed on a convex quadrilateral formed by two adjacent triangles of $T$, consisting of replacing one diagonal by the other possible diagonal. See \Cref{fig:flip}.
\end{definition}

\begin{figure}[htbp]
  \centering
  \includegraphics[width=0.48\linewidth]{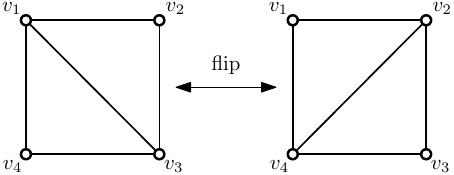}
  \caption{Example of a flip in a convex quadrilateral: the diagonal $v_1v_3$ is replaced by the diagonal $v_2v_4$.}
  \label{fig:flip}
\end{figure}

\begin{definition}[Twistable triangle]
\label{def:twistable}
Let $T$ be a triangulation of $P_N$. A triangle $\tau=\{u,v,w\}\in T$ is called \emph{twistable} if each of its three edges is a diagonal of $P_N$. We denote by $\mathcal G(T)$ the set of all twistable triangles of $T$.
\end{definition}

The following lemma gives a geometric characterization of the twistable triangles defined in \Cref{def:twistable}.

\begin{lemma}[Hexagon characterization]
\label{lem:local-hexagon}
Let $T$ be a triangulation of $P_N$, and let $\tau\in T$. Then $\tau$ is twistable if and only if it is contained in a hexagon.
\end{lemma}

\begin{proof}
Let $\tau=\{u,v,w\}$. Suppose first that $\tau$ is twistable. Then its three edges $uv$, $vw$, and $wu$ are diagonals of $P_N$. Every diagonal of a triangulation is incident with exactly two triangles. Hence, besides $\tau$, there exist unique triangles
\[
\{u,v,x\},\qquad \{v,w,y\},\qquad \{w,u,z\}\in T
\]
incident with $uv$, $vw$, and $wu$, respectively.

The vertices $x,y,z$ are distinct and lie on the sides of the three edges opposite to the interior of $\tau$. Consequently, the six vertices occur in the cyclic order $u,x,v,y,w,z$ and determine a convex hexagon $H$. The restriction of $T$ to $H$ consists precisely of the four triangles
\[
\{u,v,w\},\qquad \{u,v,x\},\qquad \{v,w,y\},\qquad \{w,u,z\}.
\]
Thus, $\tau$ is contained in a hexagon.

Conversely, suppose that $\tau$ is contained in a hexagon in this way. Each of its three edges is shared with one of the three surrounding triangles. Therefore, none of these edges belongs to the boundary of $P_N$. Hence all three edges of $\tau$ are diagonals of $P_N$, and $\tau$ is twistable.
\end{proof}

The hexagon obtained in \Cref{lem:local-hexagon} is called the \emph{hexagon associated with $\tau$}.

\begin{lemma}[Valid triangulations of the associated hexagon]
\label{lem:two-hex}
Let $T$ be a valid triangulation of $P_N$, let $\tau\in\mathcal G(T)$, and let $H$ be the hexagon associated with $\tau$. Then $H$ admits exactly two valid triangulations.
\end{lemma}

\begin{proof}
Write $\tau=\{u,v,w\}$. By \Cref{lem:local-hexagon}, the vertices of $H$ occur in the cyclic order $u,x,v,y,w,z$. Up to a permutation of the colors, assume that
\[
c(u)=A,\qquad c(v)=B,\qquad c(w)=C.
\]
Since the three triangles surrounding $\tau$ are valid, we obtain
\[
c(x)=C,\qquad c(y)=A,\qquad c(z)=B.
\]

Hence the only admissible diagonals of $H$ are $uv, vw, wu$ and $xy, yz, zx$. These diagonals form exactly two maximal noncrossing sets, corresponding to the central triangles $\{u,v,w\}$ and $\{x,y,z\}$. The two resulting triangulations are displayed in \Cref{fig:hexagon-twist}.
\end{proof}

\begin{remark}
An uncolored convex hexagon admits
\[
C_4=\frac{1}{5}\binom{8}{4}=14
\]
triangulations. Under the coloring forced by the validity condition, only the two triangulations of \Cref{lem:two-hex} remain valid, as illustrated in \Cref{fig:hexagon-twist}.
\end{remark}

\begin{figure}[htbp]
  \centering
  \includegraphics[width=0.68\linewidth]{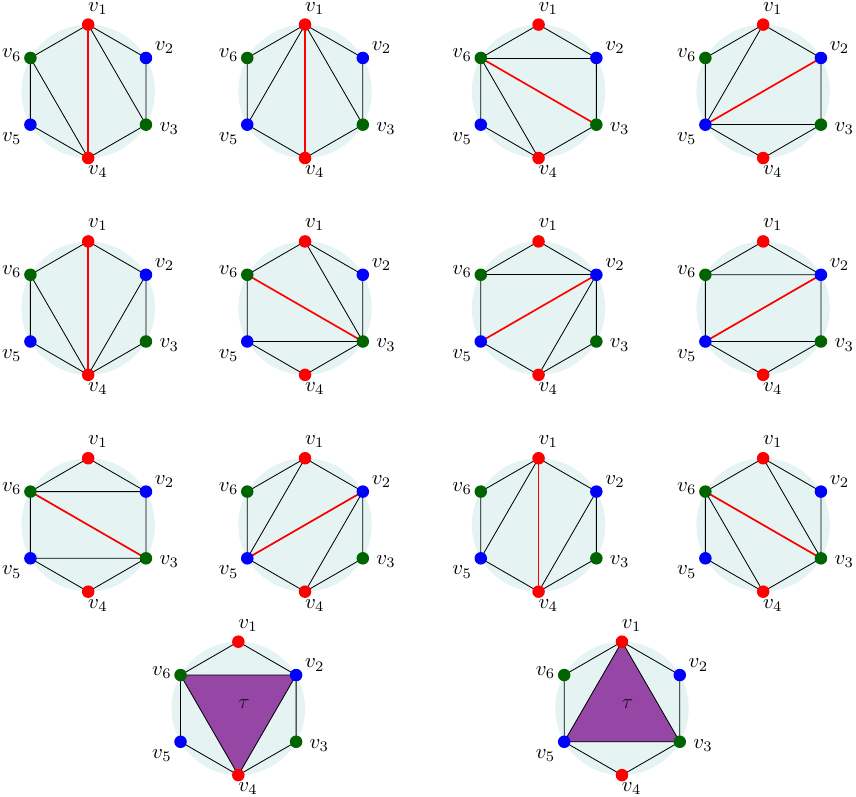}
  \caption{The $14$ triangulations of a convex hexagon. A red diagonal represents a monochromatic edge and therefore certifies that the corresponding triangulation is invalid. In each of the two valid triangulations, the twistable triangle is highlighted in purple.}
  \label{fig:hexagon-twist}
\end{figure}

Acharya, M\"utze, and Verciani~\cite{AcharyaMutzeVerciani2025} introduced the twist operation and the corresponding twist graph $\mathcal H_N$. Motivated by their construction, we make the local transformation underlying the twist explicit through the operator defined in \Cref{def:localtwist}.

\begin{definition}[Local twist]
\label{def:localtwist}
Let $T$ be a valid triangulation of $P_N$, let $\tau\in\mathcal G(T)$, and let $H$ be the hexagon associated with $\tau$. By \Cref{lem:two-hex}, the restriction of $T$ to $H$ is one of exactly two valid triangulations.

The \emph{local twist} at $\tau$ is the operation
\[
\Psi(T|_H,\tau)=(T'|_H,\tau'),
\]
where $T'$ is obtained from $T$ by replacing its restriction to $H$ with the other valid triangulation of $H$, and $\tau'$ is the central triangle of the new triangulation of $H$.
\end{definition}

\begin{definition}[Twist graph]
\label{def:twist-graph}
The \emph{twist graph} $\mathcal H_N$ is the graph whose vertices are the valid triangulations of $P_N$. Two valid triangulations $T,T'\in V(\mathcal H_N)$ are adjacent if and only if there exists a twistable triangle $\tau\in\mathcal G(T)$ such that, writing $H$ for the hexagon associated with $\tau$,
\[
\Psi(T|_H,\tau)=(T'|_H,\tau')
\]
for some $\tau'\in\mathcal G(T')$.
\end{definition}

\Cref{fig:graph-twist} illustrates this definition for $N=12$. Each vertex of the graph represents a valid triangulation of $P_{12}$, and two vertices are joined by an edge precisely when the corresponding triangulations differ by a single local twist.

\begin{figure}[htbp]
  \centering
  \includegraphics[width=\linewidth]{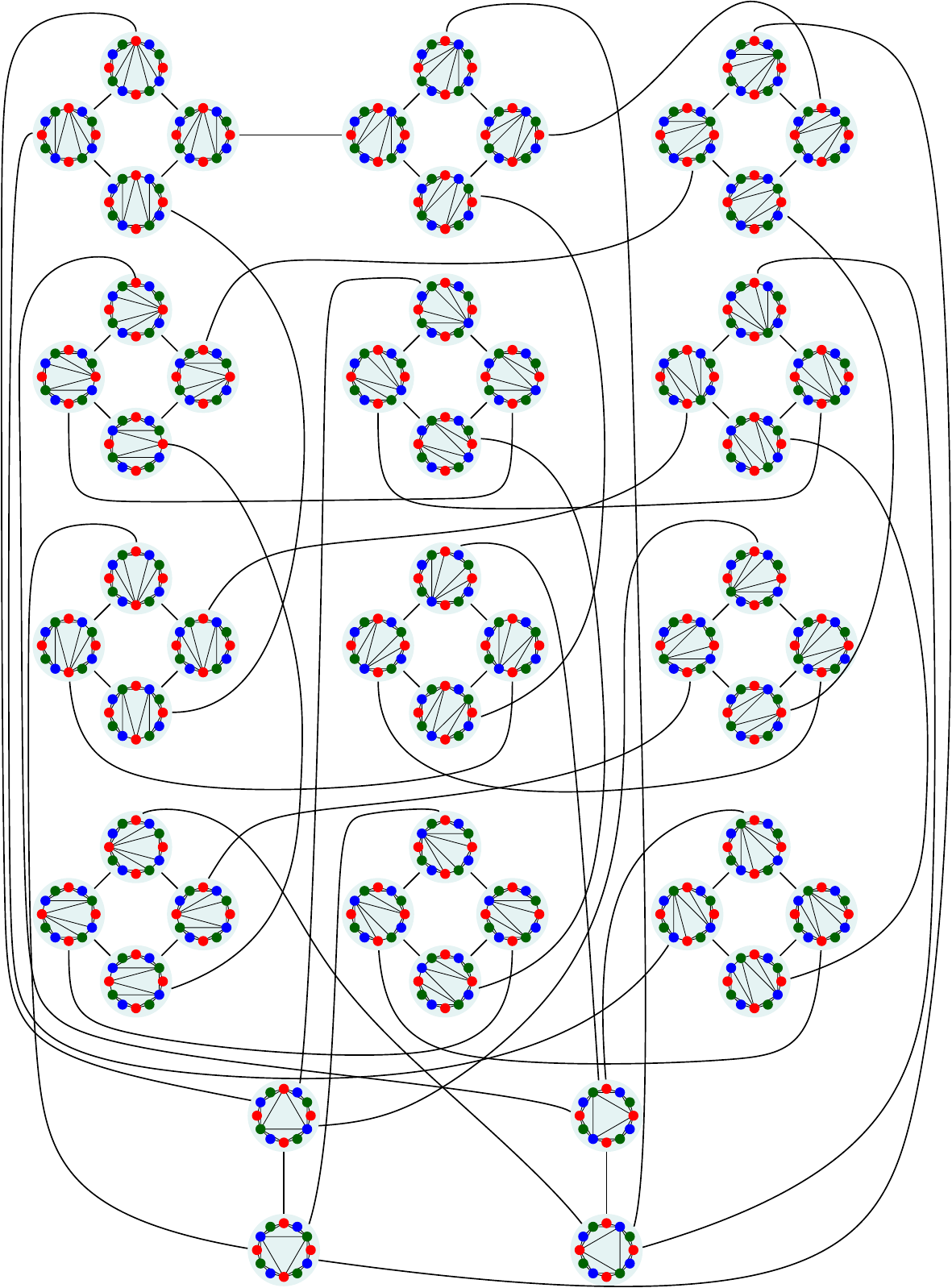}
  \caption{The twist graph $\mathcal H_{12}$. Each vertex represents a valid triangulation of the cyclically $3$-colored polygon $P_{12}$, and each edge represents a local twist between the corresponding triangulations.}
  \label{fig:graph-twist}
\end{figure}

\begin{theorem}[Acharya--M\"utze--Verciani \cite{AcharyaMutzeVerciani2025}]\label{thm:amv}
If $3\mid N$, then $\mathcal H_N$ is connected.
\end{theorem}

\section{The base-triangle decomposition}
\label{sec:base-decomposition}

This section adapts the classical root-edge decomposition underlying Segner's recurrence to the periodic $j$-coloring. In the ordinary Catalan decomposition, the third vertex of the triangle incident with a fixed boundary edge may be chosen arbitrarily. In the present setting, validity forces the color of this third vertex and, consequently, its congruence class modulo $j$.

We call the triangle incident with the fixed boundary edge the \emph{base triangle}, emphasizing its role in the decomposition of both the set of valid triangulations and the corresponding reconfiguration graph.

\begin{definition}[Base triangle and base classes]
\label{def:base}
Let $P_N=(v_1,\ldots,v_N)$ be a cyclically $j$-colored polygon, with $j\ge3$, and let $\mathcal T_N^{(j)}$ denote the set of its valid triangulations.

For every $T\in\mathcal T_N^{(j)}$, the boundary edge $v_1v_N$ belongs to a unique triangle. We write
\[
\Delta(T):=\{v_1,v_m,v_N\}
\]
and call it the \emph{base triangle} of $T$. The vertex $v_m$ is called the \emph{base vertex}.

An index $m\in\{2,\ldots,N-1\}$ is called \emph{admissible} if
\[
c(v_m)\notin\{c(v_1),c(v_N)\}.
\]

For each admissible index $m$, define
\[
\mathcal B_m := \left\{ T\in\mathcal T_N^{(j)} : \Delta(T)=\{v_1,v_m,v_N\} \right\}.
\]
Then
\[
\mathcal T_N^{(j)} = \bigsqcup_{m\ \mathrm{admissible}}\mathcal B_m.
\]
\end{definition}

For the cyclic \(3\)-coloring of
\Cref{def:cyclic-coloring} with \(N=3k+2\), the admissible
base vertices are precisely
\[
v_3,v_6,\ldots,v_{3k}.
\]
Write \(w_i:=v_{3i}\) for \(i\in\{1,\ldots,k\}\).

\begin{definition}[Base index and base-index classes]
\label{def:classes}
Let $T\in V(\mathcal H_N)$. The \emph{base index} of $T$ is the unique $i\in\{1,\ldots,k\}$ such that $\Delta(T)=\{v_1,v_N,w_i\}$. We denote it by $\ind(T):=i$.

For each $i\in\{1,\ldots,k\}$, define
\[
\mathcal S_i := \{T\in V(\mathcal H_N) : \ind(T)=i\}.
\]
Equivalently, $\mathcal S_i=\mathcal B_{3i}$, and therefore $V(\mathcal H_N)=\bigsqcup_{i=1}^{k}\mathcal S_i$.
\end{definition}

\begin{figure}[htbp]
  \centering
  \includegraphics[width=0.68\linewidth]{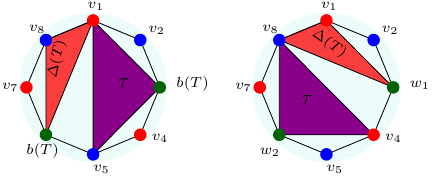}
  \caption{The base triangle in a valid triangulation $T$ of a cyclically $3$-colored polygon $P_N$, where $N=3k+2$. The fixed boundary edge $v_1v_N$ belongs to a unique triangle $\Delta(T)=\{v_1,v_N,w_i\}$, whose third vertex $w_i$ determines the base index of $T$. The base triangle is highlighted in red, while a twistable triangle is highlighted in purple.}
  \label{fig:base-triangle}
\end{figure}

Note that fixing the base triangle partitions the vertex set of $\mathcal H_N$ into pairwise disjoint classes. Indeed, every valid triangulation contains a unique triangle incident with the boundary edge $v_1v_N$, and hence has a unique base triangle. Moreover, the base triangle is not twistable: one of its edges is the boundary edge $v_1v_N$, whereas a twistable triangle must have all three of its edges among the diagonals of $P_N$. Consequently, the base triangle can never be the central triangle of a local twist and therefore cannot move by itself.

The base triangle, its base vertex, and a twistable triangle are illustrated in \Cref{fig:base-triangle}.

\subsection{The structure inside a fixed base-index class}

Fixing the base index fixes the base triangle and its two diagonal edges. These diagonals separate the remaining part of the polygon into two independently triangulated subpolygons. This yields the following product decomposition.

\begin{lemma}[Product structure]
\label{lem:product}
For every $i\in\{1,\ldots,k\}$,
\[
\mathcal H_N[\mathcal S_i] \simeq \mathcal H_{3i}\square \mathcal H_{3(k-i+1)}.
\]
In particular, $\mathcal H_N[\mathcal S_i]$ is connected.
\end{lemma}

\begin{proof}
Fix $i\in\{1,\ldots,k\}$ and write $w_i=v_{3i}$ and $\Delta_i:=\{v_1,v_N,w_i\}$. For every $T\in\mathcal S_i$, we have $\Delta(T)=\Delta_i$. Consequently, the two edges $v_1w_i$ and $w_iv_N$ belong to $T$.

Together with the boundary of $P_N$, these two diagonals separate the complement of $\Delta_i$ into the subpolygons
\[
P_i^-=(v_1,v_2,\ldots,v_{3i}) \quad\text{and}\quad P_i^+=(v_{3i},v_{3i+1},\ldots,v_N).
\]
Their numbers of vertices are $|V(P_i^-)|=3i$ and $|V(P_i^+)| = N-3i+1 = 3k+2-3i+1 = 3(k-i+1)$. In particular, both orders are divisible by $3$.

The colorings inherited by $P_i^-$ and $P_i^+$ agree with the canonical periodic $3$-coloring, possibly after a cyclic permutation of the color names. Since a global permutation of the colors does not affect validity or the twist relation, the corresponding reconfiguration graphs are isomorphic to $\mathcal H_{3i}$ and $\mathcal H_{3(k-i+1)}$, respectively.

For each $T\in\mathcal S_i$, let $T^-:=T|_{P_i^-}$ and $T^+:=T|_{P_i^+}$ be the restrictions of $T$ to the two subpolygons. Define
\[
\Phi_i:\mathcal S_i \longrightarrow V(\mathcal H_{3i}) \times V(\mathcal H_{3(k-i+1)}) \quad\text{by}\quad \Phi_i(T):=(T^-,T^+).
\]

Every triangle of $T^-$ or $T^+$ is a triangle of $T$. Therefore, both restrictions are valid. Conversely, let $R^-$ be a valid triangulation of $P_i^-$ and let $R^+$ be a valid triangulation of $P_i^+$. Since their interiors are disjoint, the union $R^-\cup R^+\cup\{\Delta_i\}$ is a triangulation of $P_N$. The triangle $\Delta_i$ is valid, and all the other triangles belong either to $R^-$ or to $R^+$. Therefore, this union is a valid triangulation belonging to $\mathcal S_i$. 

We next verify that $\Phi_i$ preserves adjacency. Let $T,T'\in \mathcal S_i$ differ by a single local twist. Since both triangulations have the same base triangle $\Delta_i$, this twist does not change the triangle incident with $v_1v_N$.

The supporting hexagon cannot have $\Delta_i$ as one of its affected triangles. Indeed, replacing the triangulation inside such a hexagon would change the triangle incident with $v_1v_N$, and the resulting triangulation would not remain in $\mathcal S_i$. Therefore, the supporting hexagon lies entirely in one of the two subpolygons $P_i^-$ or $P_i^+$.

It follows that the twist changes exactly one coordinate of $\Phi_i(T)=(T^-,T^+)$. If the twist is supported in $P_i^-$, then $\Phi_i(T')=((T^-)^\prime,T^+)$, where $T^-$ and $(T^-)^\prime$ are adjacent in $\mathcal H_{3i}$. Similarly, if the twist is supported in $P_i^+$, then $\Phi_i(T')=(T^-,(T^+)^\prime)$, where $T^+$ and $(T^+)^\prime$ are adjacent in $\mathcal H_{3(k-i+1)}$.

Conversely, a local twist in $P_i^-$ extends to a local twist of $P_N$ by leaving $T^+$ and $\Delta_i$ fixed. The analogous statement holds for a local twist in $P_i^+$. Therefore, two triangulations in $\mathcal S_i$ are adjacent if and only if their images under $\Phi_i$ agree in one coordinate and are adjacent in the other.

Hence, $\mathcal H_N[\mathcal S_i] \simeq \mathcal H_{3i}\square \mathcal H_{3(k-i+1)}$.

By \Cref{thm:amv}, both factors are connected because their orders are divisible by $3$. The Cartesian product of two connected graphs is connected. Therefore, $\mathcal H_N[\mathcal S_i]$ is connected.
\end{proof}

\subsection{Connections between distinct base-index classes}

Although the base triangle is fixed within each induced
subgraph \(\mathcal H_N[\mathcal S_i]\), a local twist may
change it when the base triangle is one of the outer triangles
of the supporting hexagon. Such twists produce edges between
distinct base-index classes.

Contracting each induced subgraph
\(\mathcal H_N[\mathcal S_i]\) to a single vertex yields the
quotient graph \(Q_k\), with
\[
V(Q_k)=\{\mathcal S_1,\ldots,\mathcal S_k\}.
\]
Two classes \(\mathcal S_p\) and \(\mathcal S_q\) are adjacent
in \(Q_k\) if there exist triangulations
\[
T_p\in\mathcal S_p
\qquad\text{and}\qquad
T_q\in\mathcal S_q
\]
that differ by a single local twist.

\begin{lemma}[Twists between base-triangle classes]
\label{lem:interclass-twists}
Let \(N=3k+2\), and let
\[
\mathcal S_1,\ldots,\mathcal S_k
\]
be the base-triangle classes associated with the admissible
third vertices
\[
w_p=v_{3p},\qquad 1\le p\le k.
\]
If
\[
1\le p<q\le k
\qquad\text{and}\qquad
q\ge p+2,
\]
then \(\mathcal S_p\) and \(\mathcal S_q\) are adjacent in
\(Q_k\).
\end{lemma}

\begin{proof}
Let
\[
w_p=v_{3p}
\qquad\text{and}\qquad
w_q=v_{3q},
\]
where \(q\ge p+2\). Choose the auxiliary vertices
\[
x=v_{3p+2}
\qquad\text{and}\qquad
y=v_{3p+4}.
\]
Since \(q\ge p+2\), we have
\[
3p<3p+2<3p+4<3q.
\]
Consequently, the six vertices
\[
v_N,\ v_1,\ w_p,\ x,\ y,\ w_q
\]
occur in this cyclic order in \(P_N\).

Their colors, in the same order, are
\[
B,A,C,B,A,C.
\]
Thus, up to reversing the cyclic order of the colors, these
vertices determine a cyclically \(3\)-colored hexagon.

We now verify that the regions between consecutive selected
vertices can be completed with valid triangulations. The
nontrivial subpolygons determined by the boundary arcs have,
respectively,
\[
3p,\qquad 3,\qquad 3,\qquad
3(q-p-1),\qquad 3(k-q+1)
\]
vertices. Indeed,
\[
\begin{aligned}
|P[v_1,w_p]|&=3p,\\
|P[w_p,x]|&=3,\\
|P[x,y]|&=3,\\
|P[y,w_q]|&=3q-(3p+4)+1
            =3(q-p-1),\\
|P[w_q,v_N]|&=N-3q+1
             =3(k-q+1).
\end{aligned}
\]
All these numbers are divisible by \(3\). Moreover,
\(q\ge p+2\) guarantees that
\[
3(q-p-1)\ge3,
\]
so none of the required regions is degenerate. By Theorem~\ref{thm:amv}, each corresponding twist graph is
nonempty and connected. In particular, each region admits a
valid triangulation. Fix one such triangulation in each region.

The remaining region is the hexagon with vertices
\[
v_N,\ v_1,\ w_p,\ x,\ y,\ w_q.
\]
Its two valid local triangulations are related by a twist. In
one of them, the triangle incident with the distinguished
boundary edge \(v_Nv_1\) has third vertex \(w_p\); in the
other, it has third vertex \(w_q\). Therefore, there exist
triangulations in \(\mathcal S_p\) and \(\mathcal S_q\) that
differ by a single twist. Hence,
\[
\mathcal S_p\mathcal S_q\in E(Q_k).
\]
\end{proof}
\begin{theorem}
\label{thm:conn-3k2}
Let \(N=3k+2\), with \(N\ge14\). Then the twist graph
\(\mathcal H_N\) is connected.
\end{theorem}

\begin{proof}
Since
\[
N=3k+2\ge14,
\]
we have \(k\ge4\). Thus, the base-triangle partition contains
at least the four classes
\[
\mathcal S_1,\mathcal S_2,\mathcal S_3,\mathcal S_4.
\]

By Lemma~\ref{lem:product}, each induced subgraph
\[
\mathcal H_N[\mathcal S_p],
\qquad 1\le p\le k,
\]
is connected. It therefore remains only to show that the
quotient graph \(Q_k\) is connected.

Taking \(p=1\) in
Lemma~\ref{lem:interclass-twists}, we obtain
\[
\mathcal S_1\mathcal S_q\in E(Q_k)
\qquad\text{for every }3\le q\le k,
\]
because
\[
q\ge3=1+2.
\]
Hence, every class except possibly \(\mathcal S_2\) is adjacent
to \(\mathcal S_1\).

Since \(k\ge4\), the class \(\mathcal S_4\) exists. Taking
\(p=2\) and \(q=4\), we have
\[
4=2+2,
\]
and the same lemma gives
\[
\mathcal S_2\mathcal S_4\in E(Q_k).
\]
Consequently, \(\mathcal S_2\) is connected to
\(\mathcal S_1\) through the path
\[
\mathcal S_2-\mathcal S_4-\mathcal S_1.
\]

Therefore, \(Q_k\) contains the spanning tree with edge set
\[
\bigl\{
\mathcal S_1\mathcal S_q:3\le q\le k
\bigr\}
\cup
\bigl\{
\mathcal S_2\mathcal S_4
\bigr\}.
\]
Thus, \(Q_k\) is connected. Since every class
\(\mathcal S_p\) induces a connected subgraph, it follows that
\(\mathcal H_N\) is connected.
\end{proof}

\begin{proposition}[Exceptional cases]
\label{prop:exceptions}
The twist graphs \(\mathcal H_8\) and
\(\mathcal H_{11}\) are disconnected.
\end{proposition}

\begin{proof}
For \(N=8\), we have
\[
8=3\cdot2+2,
\]
so there are only two base-triangle classes,
\[
\mathcal S_1
\qquad\text{and}\qquad
\mathcal S_2,
\]
associated with the consecutive admissible third vertices
\[
w_1=v_3
\qquad\text{and}\qquad
w_2=v_6.
\]

A twist changing the base triangle would require a supporting
hexagon containing, in cyclic order,
\[
v_N,\ v_1,\ w_1,\ x,\ y,\ w_2,
\]
where \(x\) and \(y\) have colors \(B\) and \(A\),
respectively. However, the only vertices strictly between
\(w_1=v_3\) and \(w_2=v_6\) are
\[
v_4,\ v_5,
\]
whose colors occur in the order \(A,B\). Thus, they cannot
complete the required cyclically colored hexagon. Hence there
is no twist between \(\mathcal S_1\) and \(\mathcal S_2\), and
\(\mathcal H_8\) is disconnected, as illustrated in
\cref{fig:H8}.

For \(N=11\), we have
\[
11=3\cdot3+2,
\]
and hence three base-triangle classes
\[
\mathcal S_1,\mathcal S_2,\mathcal S_3.
\]
By Lemma~\ref{lem:interclass-twists}, there is a twist between
\(\mathcal S_1\) and \(\mathcal S_3\).

The class \(\mathcal S_2\), however, is associated with the
third vertex \(w_2\), which is consecutive to both \(w_1\) and
\(w_3\) in the ordered list of admissible third vertices. As
in the \(N=8\) case, the boundary interval between two
consecutive admissible third vertices does not contain the two
colors in the order required to complete a supporting
cyclically colored hexagon. Therefore, no base-changing twist
joins \(\mathcal S_2\) to either \(\mathcal S_1\) or
\(\mathcal S_3\).

Thus, the quotient graph has one component containing
\(\mathcal S_1\) and \(\mathcal S_3\), while
\(\mathcal S_2\) is isolated. Consequently,
\(\mathcal H_{11}\) is disconnected, as illustrated in
\cref{fig:H11}.
\end{proof}

\begin{figure}[htbp]
  \centering
  \includegraphics[width=0.4\linewidth]{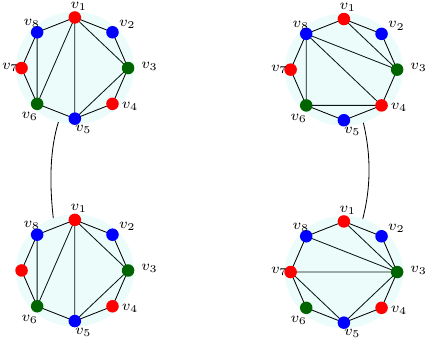}
  \caption{The twist graph \(\mathcal H_8\). The two
  base-triangle classes belong to distinct connected
  components.}
  \label{fig:H8}
\end{figure}

\begin{figure}[htbp]
  \centering
  \includegraphics[width=0.9\linewidth]{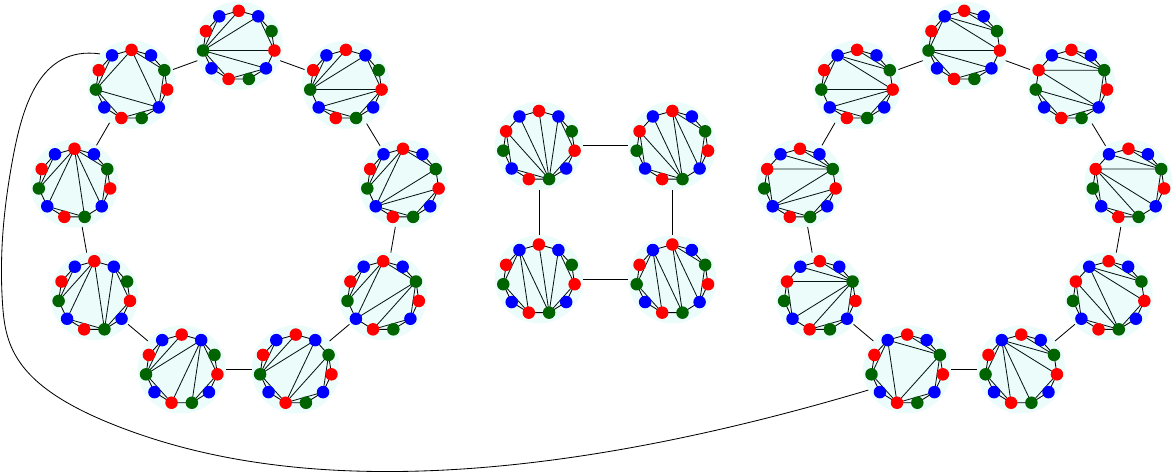}
  \caption{The twist graph \(\mathcal H_{11}\). One connected
  component contains the classes \(\mathcal S_1\) and
  \(\mathcal S_3\), while the other corresponds to
  \(\mathcal S_2\).}
  \label{fig:H11}
\end{figure}

\section{Coupled recurrences and enumeration}
\label{sec:enumeration}

Let $T(N)$ be the number of valid triangulations of $P_N$. We set $T(2)=1$ and $T(3)=1$. Define $U_k:=T(3k+2)$ for $k\ge0$ and $V_k:=T(3k)$ for $k\ge1$.

\begin{proposition}[Coupled recurrences]
\label{prop:coupled}
Let $T(N)$ denote the number of valid triangulations of the cyclically $3$-colored polygon $P_N$, with initial condition $T(3)=1$.

If $N=3k$, where $k\ge2$, then
\[
T(N) = 2T(N-1) + \sum_{i=1}^{k-2} T((N-1)-3i)\,T(3i+2).
\]

If $N=3k+2$, where $k\ge1$, then
\[
T(N) = T(N-2) + \sum_{i=1}^{k-1} T((N-2)-3i)\,T(3(i+1)).
\]
\end{proposition}

\begin{proof}
We apply the root-edge decomposition to the fixed boundary edge $v_1v_N$. Every triangulation contains a unique triangle incident with this edge. Once its third vertex is fixed, the other two sides of this triangle separate the remaining region into two subpolygons, which may be triangulated independently.

\medskip
\noindent\textbf{The case $N=3k$.}
Assume that $k\ge2$. The endpoints $v_1$ and $v_{3k}$ have colors $A$ and $C$, respectively. Since the triangle incident with $v_1v_{3k}$ must be trichromatic, its third vertex must have color $B$. Therefore, this vertex is uniquely of the form $v_{3i+2}$ for some $0\le i\le k-1$.

Fix such an index $i$. The triangle $\{v_1,v_{3i+2},v_{3k}\}$ separates the polygon into two subpolygons. The first one has $3i+2$ vertices, while the second one has $3k-(3i+2)+1 = 3k-3i-1$ vertices.

Their inherited colorings agree, up to a cyclic permutation of the color names, with the canonical periodic $3$-colorings of polygons of the corresponding orders. Since a permutation of the color names does not affect validity, valid triangulations of the two subpolygons may be chosen independently.

When $i=0$, the third vertex is $v_2$. One side of the decomposition is degenerate, and the remaining subpolygon has $3k-1$ vertices. This case therefore contributes $T(3k-1)$. Similarly, when $i=k-1$, the third vertex is $v_{3k-1}$, and this case contributes another $T(3k-1)$.

For every interior index $1\le i\le k-2$, both subpolygons are nondegenerate and contribute $T(3i+2)\,T(3k-3i-1)$. Summing over all possible positions of the third vertex gives
\[
T(3k) = 2T(3k-1) + \sum_{i=1}^{k-2} T(3i+2)\,T(3k-3i-1).
\]
Since $3k-3i-1=3k-1-3i$, this is precisely
\[
T(3k) = 2T(3k-1) + \sum_{i=1}^{k-2} T(3k-1-3i)\,T(3i+2).
\]

\medskip
\noindent\textbf{The case $N=3k+2$.}
Now assume that $k\ge1$. The endpoints $v_1$ and $v_{3k+2}$ have colors $A$ and $B$, respectively. Hence, the third vertex of the triangle incident with the boundary edge $v_1v_{3k+2}$ must have color $C$. It is therefore uniquely of the form $v_{3i}$ for some $1\le i\le k$.

For a fixed $i$, the triangle $\{v_1,v_{3i},v_{3k+2}\}$ determines two subpolygons. The first has $3i$ vertices, and the second has $(3k+2)-3i+1 = 3(k+1-i)$ vertices.

As in the previous case, restriction to the two subpolygons and gluing along the fixed triangle give a bijection between the valid triangulations having $v_{3i}$ as the third vertex and pairs of valid triangulations of the two subpolygons. Consequently,
\[
T(3k+2) = \sum_{i=1}^{k} T(3i)\,T(3(k+1-i)).
\]

We isolate the term corresponding to $i=1$. Since $T(3)=1$, this term is $T(3)\,T(3k)=T(3k)$. For the remaining terms, replace $i$ by $i+1$. We obtain
\[
T(3k+2) = T(3k) + \sum_{i=1}^{k-1} T(3(i+1))\,T(3(k-i)).
\]
Equivalently,
\[
T(3k+2) = T(3k) + \sum_{i=1}^{k-1} T(3i+3)\,T(3k-3i),
\]
which is the second recurrence.
\end{proof}

\subsection{Closed formulas and generating functions}
\label{subsec:closed-formulas}

The recurrences of \Cref{prop:coupled} count the same objects
considered by Sagan~\cite{sagan2008}. In our notation, his closed
formulas become
\begin{equation}
\label{eq:sagan-closed-forms}
T(3k)
=
\frac{2}{3k-1}\binom{4k-3}{k-1},
\qquad
T(3k+2)
=
\frac{1}{3k+1}\binom{4k}{k},
\end{equation}
where $k\ge1$ in the first identity and $k\ge0$ in the second.

Although the closed formulas follow directly from Sagan's enumeration, the generating functions provide a compact algebraic description of the coupling between the two admissible residue classes.

Let $U(x):=\sum_{k\ge0}U_kx^k$ and $V(x):=\sum_{k\ge1}V_kx^k$.

\begin{proposition}[Functional relations]
\label{prop:functional}
The generating functions associated with the two families satisfy
\[
V(x)=xU(x)^2 \qquad\text{and}\qquad U(x)=1+\frac{V(x)^2}{x}.
\]
Consequently,
\[
U(x)=1+xU(x)^4.
\]
\end{proposition}

\begin{proof}
Recall that
\[
U_0=1,\qquad U_k=T(3k+2)\quad (k\ge1),
\qquad\text{and}\qquad
V_k=T(3k)\quad (k\ge1).
\]

We first rewrite the recurrence for $T(3k)$. For $k\ge2$,
\[
T(3k)
=
2T(3k-1)
+
\sum_{i=1}^{k-2}
T(3k-1-3i)\,T(3i+2).
\]
Since
\[
3k-1=3(k-1)+2,
\]
we have
\[
T(3k-1)=U_{k-1}.
\]
Moreover, for every $i\in\{1,\ldots,k-2\}$,
\[
T(3i+2)=U_i
\]
and
\[
T(3k-1-3i)
=
T\bigl(3(k-1-i)+2\bigr)
=
U_{k-1-i}.
\]
Therefore,
\[
V_k
=
2U_{k-1}
+
\sum_{i=1}^{k-2}U_iU_{k-1-i}.
\]

Using the formal value $U_0=1$, the two copies of $U_{k-1}$
may be written as
\[
U_0U_{k-1}
\qquad\text{and}\qquad
U_{k-1}U_0.
\]
Hence
\begin{align*}
V_k
&=
U_0U_{k-1}
+
\sum_{i=1}^{k-2}U_iU_{k-1-i}
+
U_{k-1}U_0\\
&=
\sum_{i=0}^{k-1}U_iU_{k-1-i},
\qquad k\ge2.
\end{align*}
For $k=1$, the same identity holds because
\[
V_1=T(3)=1=U_0^2.
\]
Thus,
\[
V_k=\sum_{i=0}^{k-1}U_iU_{k-1-i},
\qquad k\ge1.
\]

We now translate this convolution into a generating-function identity.
Multiplying by $x^k$ and summing over $k\ge1$, we obtain
\begin{align*}
V(x)
&=
\sum_{k\ge1}V_kx^k\\
&=
\sum_{k\ge1}
\left(
\sum_{i=0}^{k-1}U_iU_{k-1-i}
\right)x^k.
\end{align*}
Setting $n=k-1$ gives
\begin{align*}
V(x)
&=
x\sum_{n\ge0}
\left(
\sum_{i=0}^{n}U_iU_{n-i}
\right)x^n\\
&=
xU(x)^2,
\end{align*}
where the last equality follows from the Cauchy product for formal
power series.

We next consider the recurrence for $T(3k+2)$. For $k\ge1$,
\[
T(3k+2)
=
T(3k)
+
\sum_{i=1}^{k-1}
T(3k-3i)\,T(3(i+1)).
\]
In terms of $U_k$ and $V_k$, this becomes
\[
U_k
=
V_k
+
\sum_{i=1}^{k-1}V_{k-i}V_{i+1}.
\]
Since $V_1=T(3)=1$, the first term can be written as
\[
V_k=V_1V_k.
\]
Now set $j=i+1$ in the sum. Then
\[
\sum_{i=1}^{k-1}V_{k-i}V_{i+1}
=
\sum_{j=2}^{k}V_{k+1-j}V_j.
\]
Consequently,
\begin{align*}
U_k
&=
V_1V_k
+
\sum_{j=2}^{k}V_{k+1-j}V_j\\
&=
\sum_{j=1}^{k}V_jV_{k+1-j},
\qquad k\ge1.
\end{align*}

By the Cauchy product,
\[
V(x)^2
=
\sum_{m\ge2}
\left(
\sum_{j=1}^{m-1}V_jV_{m-j}
\right)x^m.
\]
Dividing by $x$ and setting $k=m-1$, we obtain
\[
\frac{V(x)^2}{x}
=
\sum_{k\ge1}
\left(
\sum_{j=1}^{k}V_jV_{k+1-j}
\right)x^k
=
\sum_{k\ge1}U_kx^k.
\]
Since $U_0=1$, the right-hand side is $U(x)-1$. Therefore,
\[
U(x)=1+\frac{V(x)^2}{x}.
\]

Finally, substituting $V(x)=xU(x)^2$ into this identity gives
\[
U(x)
=
1+\frac{\bigl(xU(x)^2\bigr)^2}{x}
=
1+xU(x)^4.
\]
\end{proof}

Thus, the two coupled recurrences are encoded by a single algebraic series. The identity $U(x)=1+xU(x)^4$ is not needed here to derive the closed formulas, since these already follow from Sagan's enumeration. Its role is to show that the mutual dependence between the classes $3k$ and $3k+2$ collapses, at the generating-function level, to a quartic functional equation.

\begin{remark}[Relation with OEIS A369472]
The two counting sequences also occur interlaced in OEIS A369472~\cite{oeisA369472}, which enumerates achiral pentagonal polyominoes in the hyperbolic tiling $\{5,\infty\}$. More precisely, if $(a_n)_{n\ge1}$ denotes this sequence, then
\[
a_{2k-1}=T(3k) \qquad\text{and}\qquad a_{2k}=T(3k+2), \qquad k\ge1.
\]
At present, this identification follows only from the equality of the closed formulas. We do not claim a direct bijection between the valid triangulations and the corresponding polyominoes.
\end{remark}

\begin{theorem}[The excess is a Raney number]\label{thm:gap}
For every $k\ge1$, $T(3k+3)-T(3k+2) = \frac{5}{4(k-1)+5}\binom{4(k-1)+5}{k-1} = R_{4,5}(k-1)$, which is a Raney number of parameters $(4,5)$.
\end{theorem}

\begin{proof}
By Sagan's formulas in \eqref{eq:sagan-closed-forms},
\begin{align*}
T(3k+3)-T(3k+2) &= \frac{2}{3k+2}\binom{4k+1}{k} - \frac{1}{3k+1}\binom{4k}{k} \\
&= \left(\frac{2(4k+1)}{(3k+2)(3k+1)} - \frac{1}{3k+1}\right)\binom{4k}{k} \\
&= \frac{5k}{(3k+2)(3k+1)}\binom{4k}{k}.
\end{align*}
Using $\binom{4k+1}{k-1} = \frac{k(4k+1)}{(3k+2)(3k+1)}\binom{4k}{k}$, we obtain
\[
T(3k+3)-T(3k+2) = \frac{5}{4k+1}\binom{4k+1}{k-1} = \frac{5}{4(k-1)+5}\binom{4(k-1)+5}{k-1}.
\]
The last expression is $R_{4,5}(k-1)$ by the definition $R_{p,r}(n)=\frac{r}{pn+r}\binom{pn+r}{n}$.
\end{proof}

\section{Cyclic colorings with \texorpdfstring{$j\ge4$}{j >= 4}}
\label{sec:more-colors}

Let $P_n=(v_1,\ldots,v_n)$ be a convex polygon whose vertices are cyclically colored with $j\ge4$ colors. Thus, the color of $v_q$ is determined by the residue class of $q$ modulo $j$.

A triangulation is called \emph{valid} if the three vertices of each of its triangles have pairwise distinct colors. Let $t(n)$ denote the number of valid triangulations of $P_n$. As usual, we use the formal convention $t(2)=1$ for a degenerate subpolygon arising at an extreme position of the base triangle.

\begin{lemma}[General root-edge recurrence]
\label{lem:recurrence-general}
Let $n=jk+r$, $k\ge1$, $0\le r\le j-1$, $r\ne1$.

If $2\le r\le j-1$, then
\[
t(n) = \sum_{i=0}^{k} \sum_{\ell=1}^{r-2} t(n-\ell-ji)\, t(\ell+1+ji) + \sum_{i=0}^{k-1} \sum_{\ell=r}^{j-1} t(n-\ell-ji)\, t(\ell+1+ji). \tag{\(\star\)}\label{eq:recurrence-general-r}
\]

If $r=0$, then
\[
t(jk) = \sum_{i=0}^{k-1} \sum_{\ell=1}^{j-2} t(jk-\ell-ji)\, t(\ell+1+ji). \tag{\(\star\star\)}\label{eq:recurrence-general-zero}
\]
\end{lemma}

\begin{proof}
Fix the boundary edge $v_1v_n$. Every valid triangulation of $P_n$ contains a unique triangle incident with this edge. Write this base triangle as $\{v_1,v_m,v_n\}$, $2\le m\le n-1$.

Once $v_m$ is fixed, the other two sides of the base triangle divide $P_n$ into two subpolygons: $(v_1,v_2,\ldots,v_m)$ and $(v_m,v_{m+1},\ldots,v_n)$. These subpolygons have, respectively, $m$ and $n-m+1$ vertices.

Restriction and gluing give a bijection between the valid triangulations having $\{v_1,v_m,v_n\}$ as their base triangle and pairs of valid triangulations of these two subpolygons. Therefore, the contribution of the base vertex $v_m$ is $t(m)t(n-m+1)$.

It remains to determine the admissible values of $m$. Write $m=\ell+1+ji$, $0\le\ell\le j-1$. The residue $\ell$ records the position of $v_m$ within a period of the cyclic coloring.

Since $m\equiv\ell+1\pmod j$, the vertices $v_m$ and $v_1$ have the same color if and only if $\ell=0$. Moreover, because $n\equiv r\pmod j$, the vertices $v_m$ and $v_n$ have the same color if and only if $\ell\equiv r-1\pmod j$. Thus, the two forbidden residues are $\ell=0$ and $\ell\equiv r-1\pmod j$.

Suppose first that $2\le r\le j-1$. The admissible residues are then $\ell\in \{1,\ldots,r-2\} \cup \{r,\ldots,j-1\}$.

For $1\le\ell\le r-2$, the inequality $2\le \ell+1+ji\le n-1$ allows $0\le i\le k$. For $r\le\ell\le j-1$, the last incomplete period contains no further occurrence, and hence $0\le i\le k-1$.

For a given pair $(i,\ell)$, we have $m=\ell+1+ji$ and $n-m+1 = n-(\ell+1+ji)+1 = n-\ell-ji$. Consequently, the corresponding contribution is $t(\ell+1+ji)\,t(n-\ell-ji)$. Summing over all admissible pairs $(i,\ell)$ gives \eqref{eq:recurrence-general-r}.

Now suppose that $r=0$. In this case, $v_n$ has the color corresponding to residue $0$ modulo $j$. The forbidden residues for $\ell$ are therefore $\ell=0$ and $\ell=j-1$. Thus, $1\le\ell\le j-2$. For every such $\ell$, the possible values of $i$ are $0\le i\le k-1$. The same restriction-and-gluing argument then gives \eqref{eq:recurrence-general-zero}.
\end{proof}

\begin{remark}[Relation with Sagan's enumerative framework]
\label{rem:sagan-general}
Sagan~\cite{sagan2008} studied proper \(k\)-partitions of cyclically colored polygons. Under his original definition, every \(k\)-gon must contain all \(c\) colors among its vertices. He observed that the enumeration already becomes difficult outside the cases admitting generalized Catalan formulas. For example, when \(c=3\) and \(k=4\), the resulting recurrences do not appear to yield simple generating functions.

Sagan also considered the case \(c>k\), for which the original definition cannot apply, and suggested instead requiring every \(k\)-gon to contain \(k\) distinct vertex colors. He reported that the case \(c=4\) and \(k=3\) leads to similar algebraic difficulties. Our validity condition agrees with this proposed definition only in the triangulation specialization \(k=3\), where \(c=j\).

Consequently, \Cref{lem:recurrence-general} should not be regarded as a solution of Sagan's broader enumeration problem for arbitrary proper \(k\)-partitions. It provides instead a uniform recursive description of the triangulation case for every fixed number \(j\) of colors. The following \(j=4\) specialization illustrates how this recurrence can be converted into a finite algebraic system, while also showing why the dependence on the residue classes modulo \(j\) becomes increasingly complicated as \(j\) grows.
\end{remark}

\begin{remark}[The case $j=4$]
For every fixed $j$, separating the counting sequence according to the residue classes modulo $j$ transforms the recurrence of \Cref{lem:recurrence-general} into a finite system of algebraic functional equations. For $j=3$, this system reduces to the quartic equation obtained in \Cref{prop:functional}.

We describe the analogous reduction for $j=4$. Define the shifted sequence
\[
a_n:=t(n+2),\qquad n\ge0,
\]
so that $a_n$ counts the valid triangulations of a polygon with $n+2$ vertices. Since no valid triangulation exists when the number of vertices is congruent to $1$ modulo $4$, we have
\[
a_{4n+3}=t(4n+5)=0,\qquad n\ge0.
\]

Let
\[
A(x):=\sum_{n\ge0}a_nx^n
\]
be the ordinary generating function of the complete counting sequence. We separate this sequence into its nonzero residue classes by defining
\[
A_r(x):=\sum_{n\ge0}a_{4n+r}x^n,
\qquad r\in\{0,1,2\}.
\]
Thus,
\[
A(x) = A_0(x^4)+xA_1(x^4)+x^2A_2(x^4).
\]

Applying the recurrence of \Cref{lem:recurrence-general} separately to the three admissible residue classes gives
\begin{align*}
a_{4n} &= 2\sum_{\substack{p,q\ge0\\p+q=n-1}} a_{4p+1}a_{4q+2}, && n\ge1,\\
a_{4n+1} &= \sum_{\substack{p,q\ge0\\p+q=n}} a_{4p}a_{4q} + \sum_{\substack{p,q\ge0\\p+q=n-1}} a_{4p+2}a_{4q+2}, && n\ge0,\\
a_{4n+2} &= 2\sum_{\substack{p,q\ge0\\p+q=n}} a_{4p}a_{4q+1}, && n\ge0.
\end{align*}
As usual, a sum over an empty index set is interpreted as zero.

Since $a_0=t(2)=1$, $a_1=t(3)=1$, and $a_2=t(4)=2$, the corresponding generating functions satisfy
\begin{align}
A_0(x)&=1+2xA_1(x)A_2(x),\label{eq:j4-A0}\\
A_1(x)&=A_0(x)^2+xA_2(x)^2,\label{eq:j4-A1}\\
A_2(x)&=2A_0(x)A_1(x).\label{eq:j4-A2}
\end{align}

We now reduce this coupled system to a single equation. Substituting \eqref{eq:j4-A2} into \eqref{eq:j4-A0} gives
\[
A_0-1 = 2xA_1(2A_0A_1) = 4xA_0A_1^2.
\]
Consequently,
\[
4xA_1^2=\frac{A_0-1}{A_0}.
\]
Using this identity and \eqref{eq:j4-A2} in \eqref{eq:j4-A1}, we obtain
\begin{align*}
A_1 &= A_0^2+x(2A_0A_1)^2 \\
&= A_0^2\bigl(1+4xA_1^2\bigr) \\
&= A_0^2 \left( 1+\frac{A_0-1}{A_0} \right) \\
&= A_0(2A_0-1).
\end{align*}

Define $Z(x):=2\bigl(A_0(x)-1\bigr)$. Equivalently,
\[
A_0(x)=1+\frac{Z(x)}{2} = \frac{2+Z(x)}{2}.
\]
It follows that $2A_0(x)-1=1+Z(x)$, and therefore
\[
A_1(x) = \frac{(1+Z(x))(2+Z(x))}{2}.
\]
Moreover, by \eqref{eq:j4-A2},
\[
A_2(x) = \frac{(1+Z(x))(2+Z(x))^2}{2}.
\]

Finally, multiplying \eqref{eq:j4-A0} by $2$ gives $Z(x)=4xA_1(x)A_2(x)$. Substituting the expressions for $A_1(x)$ and $A_2(x)$ yields
\[
\boxed{
Z(x)=x(1+Z(x))^2(2+Z(x))^3.
}
\]

Thus, the complete generating function can be recovered from $Z$ as
\[
A(x) = 1+\frac{Z(x^4)}{2} + \frac{x}{2}\bigl(1+Z(x^4)\bigr)\bigl(2+Z(x^4)\bigr) + \frac{x^2}{2}\bigl(1+Z(x^4)\bigr)\bigl(2+Z(x^4)\bigr)^2.
\]

The equation for $Z(x)$ is in the standard form $Z=x\Phi(Z)$, where $\Phi(u)=(1+u)^2(2+u)^3$. Hence, by the Lagrange--B\"urmann inversion formula, for every $n\ge1$,
\[
[x^n]Z(x) = \frac{1}{n} [u^{n-1}] (1+u)^{2n}(2+u)^{3n}.
\]
Together with the expressions above for $A_0$, $A_1$, and $A_2$, this gives explicit coefficient formulas for each admissible residue class.

For larger values of $j$, the same decomposition produces increasingly large coupled systems. We do not know whether these systems admit a comparable reduction to a single functional equation or a uniform closed formula valid for every $j$. The separation into residue classes appears to be intrinsic, since the admissible base vertices depend on the order of the polygon modulo $j$.
\end{remark}

\begin{definition}[Flip reconfiguration graph]
\label{def:flip-reconfiguration-graph}
Let $j\ge4$ and $N\ge3$. The \emph{flip reconfiguration graph} for cyclically $j$-colored triangulations of an $N$-vertex convex polygon, denoted by $\mathcal G_N^{(j)}$, is the graph whose vertices are the valid triangulations of $P_N$. Two vertices of $\mathcal G_N^{(j)}$ are adjacent if and only if the corresponding triangulations differ by a single validity-preserving flip.
\end{definition}

\begin{theorem}[Connectedness for cyclic colorings with $j\ge4$]
\label{thm:jge4}
Let $N=jk+r$, $j\ge4$, $k\ge1$, $0\le r\le j-1$, $r\ne1$. Then the reconfiguration graph $\mathcal G_N^{(j)}$ is connected.
\end{theorem}

\begin{proof}
Fix $j\ge4$. We prove, by strong induction on $N$, the slightly stronger statement that the reconfiguration graph is connected for every $N\ge2$ such that $N\not\equiv1\pmod j$.

We regard the degenerate polygon $P_2$ as having a unique empty triangulation, so its reconfiguration graph is the one-vertex graph. When $3\le N\le j$, all vertices of $P_N$ have distinct colors. Hence every triangulation is valid, and $\mathcal G_N^{(j)}$ is the classical flip graph of a convex polygon. In particular, it is connected. When $5\le N\le j$, Lucas~\cite{lucas1987} proved the
stronger existence of a Hamilton cycle.

Now suppose that $N>j$ and that the statement holds for every smaller polygon order not congruent to $1$ modulo $j$. Fix the boundary edge $v_1v_N$. Since $N\not\equiv1\pmod j$, its endpoints have distinct colors: $\ccl(v_1)\neq\ccl(v_N)$.

By \Cref{def:base}, an index $m\in\{2,\ldots,N-1\}$ is admissible precisely when $\ccl(v_m)\notin\{\ccl(v_1),\ccl(v_N)\}$. By periodicity, this is equivalent to $m\not\equiv1\pmod j$ and $m\not\equiv N\pmod j$.

Let $m_1<m_2<\cdots<m_s$ be all the admissible indices. The corresponding base classes $\mathcal B_{m_1},\mathcal B_{m_2},\ldots,\mathcal B_{m_s}$ form a partition of $V(\mathcal G_N^{(j)})$.

We first show that each class $\mathcal B_{m_q}$ induces a connected subgraph. Fix $q\in\{1,\ldots,s\}$. The base triangle $\Delta_q:=\{v_1,v_{m_q},v_N\}$ divides $P_N$ into the two subpolygons $P_q^-=(v_1,v_2,\ldots,v_{m_q})$ and $P_q^+=(v_{m_q},v_{m_q+1},\ldots,v_N)$, whose orders are $|V(P_q^-)|=m_q$ and $|V(P_q^+)|=N-m_q+1$.

The first order is not congruent to $1$ modulo $j$, since $m_q\not\equiv1\pmod j$. The second is also not congruent to $1$ modulo $j$, since $N-m_q+1\equiv1\pmod j$ would imply $m_q\equiv N\pmod j$, contradicting the admissibility of $m_q$.

Restriction and gluing therefore give
\[
\mathcal G_N^{(j)}[\mathcal B_{m_q}] \simeq \mathcal G_{m_q}^{(j)} \square \mathcal G_{N-m_q+1}^{(j)},
\]
where a factor corresponding to a degenerate $2$-gon is understood to be the one-vertex graph. The inherited colorings agree with the canonical cyclic coloring up to a permutation of the color names.

By the induction hypothesis, both nondegenerate factors are connected. Hence $\mathcal G_N^{(j)}[\mathcal B_{m_q}]$ is connected.

It remains to connect consecutive base classes. Let $q\in\{1,\ldots,s-1\}$. The admissible indices are precisely those whose residues modulo $j$ avoid the two residues corresponding to the colors of $v_1$ and $v_N$. Since $j\ge4$, at least $j-2\ge2$ residue classes remain. Therefore, consecutive admissible indices $m_q$ and $m_{q+1}$ have distinct residues modulo $j$, and hence $\ccl(v_{m_q})\neq\ccl(v_{m_{q+1}})$.

Since the colors of $v_1$ and $v_N$ are distinct and both $v_{m_q}$ and $v_{m_{q+1}}$ have colors different from them, the four vertices $v_1, v_{m_q}, v_{m_{q+1}}, v_N$ have pairwise distinct colors. In their cyclic order, they determine the convex quadrilateral $Q_q=(v_1,v_{m_q},v_{m_{q+1}},v_N)$.

The boundary of $Q_q$ separates the remainder of $P_N$ into the three subpolygons $R_q^{(1)}=(v_1,v_2,\ldots,v_{m_q})$, $R_q^{(2)}=(v_{m_q},v_{m_q+1},\ldots,v_{m_{q+1}})$, and $R_q^{(3)}=(v_{m_{q+1}},v_{m_{q+1}+1},\ldots,v_N)$. Their orders are, respectively, $m_q$, $m_{q+1}-m_q+1$, and $N-m_{q+1}+1$.

None of these orders is congruent to $1$ modulo $j$. Indeed, this follows respectively from $m_q\not\equiv1\pmod j$, $m_{q+1}\not\equiv m_q\pmod j$, and $m_{q+1}\not\equiv N\pmod j$. Thus, by the induction hypothesis, every nondegenerate region $R_q^{(1)},R_q^{(2)},R_q^{(3)}$ admits a valid triangulation. Fix one such triangulation in each region.

Inside $Q_q$, first choose the diagonal $v_{m_q}v_N$. The two resulting triangles are $\{v_1,v_{m_q},v_N\}$ and $\{v_{m_q},v_{m_{q+1}},v_N\}$. They are valid because the four vertices of $Q_q$ have pairwise distinct colors. Together with the fixed triangulations of the three complementary regions, they determine a valid triangulation $T_q\in\mathcal B_{m_q}$.

Now flip $v_{m_q}v_N$ to the other diagonal of $Q_q$, namely $v_1v_{m_{q+1}}$. The two new triangles are $\{v_1,v_{m_q},v_{m_{q+1}}\}$ and $\{v_1,v_{m_{q+1}},v_N\}$, which are also valid. The resulting triangulation, denoted by $T'_q$, has base triangle $\{v_1,v_{m_{q+1}},v_N\}$. Therefore, $T'_q\in\mathcal B_{m_{q+1}}$. Consequently, the flip inside $Q_q$ produces an edge between $\mathcal B_{m_q}$ and $\mathcal B_{m_{q+1}}$.

We have proved that every class $\mathcal B_{m_q}$ induces a connected subgraph and that every pair of consecutive classes is joined by an edge. Hence the classes form the connected chain
\[
\mathcal B_{m_1} \longleftrightarrow \mathcal B_{m_2} \longleftrightarrow \cdots \longleftrightarrow \mathcal B_{m_s}.
\]
Since these classes partition $V(\mathcal G_N^{(j)})$, the graph $\mathcal G_N^{(j)}$ is connected.
\end{proof}



\newpage

\begin{thebibliography}{99}

\bibitem{segner1761}
J.~A.~Segner,
\emph{Enumeratio modorum quibus figurae planae rectilineae per diagonales dividuntur in triangula},
Novi Comment. Acad. Sci. Imp. Petropol. 7
(1758/59; published 1761), 203--210.

\bibitem{AcharyaMutzeVerciani2025}
R.~Acharya, T.~M\"utze, and F.~Verciani,
Flips in colorful triangulations,
\emph{J. Comput. Geom.} 16(1) (2025), 295--332,
\url{https://doi.org/10.20382/jocg.v16i1a9}.

\bibitem{devadoss1999}
S.~L.~Devadoss,
Tessellations of moduli spaces and the mosaic operad,
\emph{Contemp. Math.} 239 (1999), 91--114,
\url{https://doi.org/10.1090/conm/239/03599}.

\bibitem{hurtado1999}
F.~Hurtado, M.~Noy, and J.~Urrutia,
Flipping edges in triangulations,
\emph{Discrete Comput. Geom.} 22 (1999), 333--346,
\url{https://doi.org/10.1007/PL00009464}.

\bibitem{ItoEtAl2022}
T.~Ito, Y.~Iwamasa, Y.~Kobayashi, S.-i.~Maezawa,
Y.~Nozaki, Y.~Okamoto, and K.~Ozeki,
Reconfiguration of colorings in triangulations of the sphere,
\emph{J. Comput. Geom.} 16(1) (2025), 253--294,
\url{https://doi.org/10.20382/jocg.v16i1a8}.

\bibitem{lawson1972}
C.~L.~Lawson,
Transforming triangulations,
\emph{Discrete Math.} 3(4) (1972), 365--372,
\url{https://doi.org/10.1016/0012-365X(72)90093-3}.

\bibitem{lee1989}
C.~W.~Lee,
The associahedron and triangulations of the $n$-gon,
\emph{European J. Combin.} 10(6) (1989), 551--560,
\url{https://doi.org/10.1016/S0195-6698(89)80072-1}.

\bibitem{lucas1987}
J.~M.~Lucas,
The rotation graph of binary trees is Hamiltonian,
\emph{J. Algorithms} 8 (1987), 503--535,
\url{https://doi.org/10.1016/0196-6774(87)90048-4}.

\bibitem{mutze2023gray}
T.~M\"utze,
Combinatorial Gray codes---an updated survey,
\emph{Electron. J. Combin.} DS26 (2023), 99~pp.,
\url{https://doi.org/10.37236/11023}.

\bibitem{nishimura2018}
N.~Nishimura,
Introduction to reconfiguration,
\emph{Algorithms} 11(4) (2018), Article~52,
\url{https://doi.org/10.3390/a11040052}.

\bibitem{pournin2014}
L.~Pournin,
The diameter of associahedra,
\emph{Adv. Math.} 259 (2014), 13--42,
\url{https://doi.org/10.1016/j.aim.2014.02.035}.

\bibitem{sagan2008}
B.~E.~Sagan,
Proper partitions of a polygon and $k$-Catalan numbers,
\emph{Ars Combin.} 88 (2008), 109--124,
arXiv:math/0407280.

\bibitem{sleator1988}
D.~D.~Sleator, R.~E.~Tarjan, and W.~P.~Thurston,
Rotation distance, triangulations, and hyperbolic geometry,
\emph{J. Amer. Math. Soc.} 1(3) (1988), 647--681,
\url{https://doi.org/10.1090/S0894-0347-1988-0928904-4}.

\bibitem{stanley1999}
R.~P.~Stanley,
\emph{Enumerative Combinatorics}, Vol.~2,
Cambridge University Press, Cambridge, 1999,
\url{https://doi.org/10.1017/CBO9780511609589}.

\bibitem{stasheff1963}
J.~D.~Stasheff,
Homotopy associativity of H-spaces. I, II,
\emph{Trans. Amer. Math. Soc.} 108 (1963), 275--312,
\url{https://doi.org/10.1090/S0002-9947-1963-0158400-5}.

\bibitem{vandenheuvel2013}
J.~van den Heuvel,
The complexity of change,
in \emph{Surveys in Combinatorics 2013},
Cambridge University Press, Cambridge, 2013, 127--160,
\url{https://doi.org/10.1017/CBO9781139506748.005}.

\bibitem{oeisA369472}
OEIS Foundation Inc.,
\emph{The On-Line Encyclopedia of Integer Sequences},
Sequence A369472,
\url{https://oeis.org/A369472}.

\end{thebibliography}
\end{document}